\documentclass[12pt,reqno]{amsart}
\usepackage{amssymb}

\newcommand{\gr}{\text{Gr}(E)}
\newcommand{\str}{\mathcal{O}}
\newcommand{\proj}{\mathbb{P}}

\newtheorem{theorem}{Theorem}[section]
\newtheorem{proposition}[theorem]{Proposition}
\newtheorem{lemma}[theorem]{Lemma}
\newtheorem{corollary}[theorem]{Corollary}

\theoremstyle{definition}

\newtheorem{remark}[theorem]{Remark}
\newtheorem{example}[theorem]{Example}
\numberwithin{equation}{section}

\begin{document}

\baselineskip=15pt

\title[Seshadri constants and Grassmann bundles]{Seshadri constants and Grassmann 
bundles over curves}

\author[I. Biswas]{Indranil Biswas}

\address{School of Mathematics, Tata Institute of Fundamental
Research, Homi Bhabha Road, Mumbai 400005, India}

\email{indranil@math.tifr.res.in}

\author[K. Hanumanthu]{Krishna Hanumanthu}

\address{Chennai Mathematical Institute, H1 SIPCOT IT Park, Siruseri, Kelambakkam 603103, India}

\email{krishna@cmi.ac.in}

\author[D. S. Nagaraj]{D. S. Nagaraj}

\address{The Institute of Mathematical Sciences (HBNI), CIT
Campus, Taramani, Chennai 600113, India; Indian Institute of Science Education and
Research, Tirupati, Rami Reddy Nagar, Karakambadi Road,
Mangalam (P.O.) Tirupati 517507, India}

\email{dsn@imsc.res.in and dsn@iisertirupati.ac.in}

\author[P. E. Newstead]{Peter E. Newstead}

\address{Department of Mathematical Sciences, University of Liverpool, Peach 
Street, Liverpool, L69 7ZL, England}

\email{newstead@liverpool.ac.uk}

\subjclass[2010]{14H60, 14Q20, 14C20}

\keywords{Seshadri constant, Harder-Narasimhan filtration, Grassmann bundle, nef cone,
pseudo-effective cone.}

\date{}

\begin{abstract}
Let $X$ be a smooth complex projective curve, and let $E$ be a vector
bundle on $X$ which is not semistable. For a suitably
chosen integer $r$, let $\gr$ be 
the Grassmann bundle over $X$ that parametrizes the quotients of the fibers of 
$E$ of dimension $r$. Assuming some numerical conditions on the
Harder-Narasimhan filtration of $E$, we study Seshadri constants of
ample line bundles on $\gr$. In many cases, we give the precise values of
the Seshadri constants. Our results generalize various known results for
the special case of ${\rm rank}(E) \,=\, 2$. We include some examples in rank $4$.
\end{abstract}

\maketitle
\section{Introduction}\label{se1}

Seshadri constants were introduced by Demailly \cite{D} in 1990.
He was motivated by an ampleness criterion of Seshadri \cite[Theorem 7.1]{H}.
Seshadri constants have become an important area of research dealing with positivity
of line bundles on projective varieties. 

Let $M$ be a smooth complex projective variety, and let $L$ be a nef 
line bundle on $M$. Take any point $x \,\in\, M$. The {\it Seshadri constant of $L$
at $x$}, which is a nonnegative real number denoted by $\varepsilon(M,L,x)$, is defined
to be
$$\varepsilon(M,L,x)\,:=\, \inf\limits_{\substack{x \in C}} \frac{L\cdot
C}{{\rm mult}_{x}C}\, ,$$
where the infimum is taken over all closed curves in $M$ passing through the point $x$. Here 
$L\cdot C$ denotes the intersection multiplicity, so only the
numerical class of $L$ is relevant in the above definition, and ${\rm mult}_x C$ denotes the 
multiplicity of the curve $C$ at $x$. It is easy to check that the infimum can
equivalently be taken just over irreducible and reduced curves $C$.

Seshadri's above
mentioned criterion for ampleness says that 
$L$ is ample if and only if $\varepsilon(M,L,x) \,>\, 0$ for all $x \,\in\, M$.

If the variety $M$ is clear from the context, we simply write
$\varepsilon(L,x)$ instead of $\varepsilon(M,L,x)$.

Usually, the Seshadri constants are not easy to compute precisely
and a lot of work
has focused on giving bounds on Seshadri constants. If $L$ is an
ample line bundle on a projective variety $M$ of dimension $n$, then it is
not difficult to show
that $\varepsilon(L,x)\,\le \,\sqrt[n]{L^n}$ for all $x \,\in\, M$, where $L^n$ denotes the top
self-intersection of $L$. So $\varepsilon(L,x)$ belongs to the
interval $(0,\sqrt[n]{L^n}]$. In many specific cases, it is an 
interesting problem to shrink this range of possible values for Seshadri
constants and a lot of work has been done in this direction. For an
overview of current research on Seshadri constants, see \cite{B}.

Since the Seshadri constant of an ample line bundle $L$ at every point
$x$ is bounded above by $\sqrt[n]{L^n}$, we can consider their 
supremum as $x$ varies and define\footnote{It may be mentioned that the numeral 1 in the
definition of $\varepsilon(L,1)$ refers to the fact that we are
considering Seshadri constants at one point. Seshadri constants can also 
be defined at a finite set of points and in this generality they are
called {\it multi-point} Seshadri constants.}
$$\varepsilon(L,1) \,:=\, \sup\limits_{x\in M}
\varepsilon(L,x)\, .$$ 

We also define
$$\varepsilon(L) \,:=\, \inf\limits_{\substack{x \in M}} \varepsilon(L,x)\, .$$

When $L$ is ample, the following inequalities hold for every
$x \,\in\, M$: $$0 \,<\,
\varepsilon(L) \,\le\, \varepsilon(L,x) \,\le\, \varepsilon(L,1)
\,\le\, \sqrt[n]{L^n}\, .$$ 

An example of Miranda \cite[Example 5.2.1]{La2} shows that
$\varepsilon(L)$ can be arbitrarily
small; to be more precise, given a real
number $\delta > 0$, there exists an algebraic surface $M$ (which is
obtained by blowing up the projective plane $\mathbb{P}^2$ at suitably
chosen points) and an ample line bundle $L$ on $M$ such that
$\varepsilon(M,L) \,< \,\delta$. 

On the other hand, Ein and Lazarsfeld, \cite{EL}, have proved that 
$\varepsilon(M,L,1)\,\ge\, 1$ for any ample line bundle $L$ on any surface $M$. In 
fact, they show that the inequality $\varepsilon(M,L,x)\,< \, 1$ holds for at most countably many 
points $x \,\in\, M$. In the case of surfaces, it is known that 
$\varepsilon(M,L,1)$ is achieved at {\it very general} points on $M$ (that is, 
points outside a countable union of Zariski closed proper subsets in $M$); for 
more details, see \cite{Og}.

As the two results mentioned above suggest, most of the research on Seshadri constants is 
focused on the case of surfaces. There are many important open problems on Seshadri constants
for surfaces which attract most attention and not a lot is known for higher 
dimensional varieties.

Let $M$ be a smooth complex projective variety of dimension $n$. For
any ample line bundle $L$ on $M$, it is proved in \cite{EKL} that
$\varepsilon(L,1) \,\ge\, \frac{1}{n}$. This partially generalizes the
earlier mentioned result of Ein and Lazarsfeld for $n\,=\,2$. Some improvements
for this are known for specific classes of varieties. 
For example, if $M$ is a threefold, Cascini and Nakamaye in \cite{CN} obtained the bound
$\varepsilon(L,1) \,>\, \frac{1}{2}$.
If $M$ is a Fano variety of dimension $n \, \ge \, 3$ and its anti-canonical 
divisor is globally generated, then $\varepsilon(L) \, \ge \,
\frac{1}{n-2}$, except when $M$ is a del Pezzo three-fold of 
degree 1 (see \cite{Le}). When $M$ is a
toric variety and more generally when $M$ has toric degenerations, Ito,
\cite{It2}, has bounds on $\varepsilon(L,1)$ for an ample line bundle
$L$ on $M$. When $M$ is arbitrary, there are
bounds on $\varepsilon(L,1)$ in terms of Seshadri
constants on the toric variety associated to the Newton--Okounkov body
of $L$ \cite{It1}. If $M$ is an abelian variety, we have $\varepsilon(L,x)
\,=\,\varepsilon(L) \,=\, \varepsilon(L,1)$ for all $x \,\in\, M$, because
a translation does not change Seshadri constants. The case of abelian
varieties has been studied extensively; see, for example, \cite{Na, La1, Ba, Deb}.

Seshadri constants on ruled surfaces have been looked at by many authors (see 
\cite{F,HM,S}). Since ruled surfaces are projective bundles associated to rank two vector 
bundles over smooth projective curves, it is natural to ask if similar results can be proved for 
projective bundles of higher rank vector bundles on smooth curves. 
Our aim here is to study Seshadri constants for line bundles on
certain Grassmann bundles over smooth complex projective curves.

Let $X$ be a smooth complex projective curve, and let $E$ be a vector
bundle over $X$ of rank $n$. For a fixed integer $1\, \leq\, r\, <\, n$, let $\gr$
denote the Grassmann bundle over $X$ parametrizing all $r$-dimensional
quotients of fibers of $E$. We study Seshadri constants of
ample line bundles on $\gr$ for a suitably chosen $r$. Some numerical
assumptions on the Harder-Narasimhan filtration of $E$ are imposed
to obtain the results on Seshadri constants.

A related problem has been studied
in \cite{BSS}, \cite{Hac}. In these papers, the authors define Seshadri
constants for a vector bundle $E$ on a projective variety $X$ in 
terms of the tautological line bundle on the projective bundle
$\proj(E)$ over $X$. They also give some bounds on these Seshadri
constants. We remark that this is different from directly considering Seshadri
constants of line bundles on $\proj(E)$.

In Section \ref{grassmann}, we investigate the geometry of $\gr$ and 
obtain the results required for the computation of
Seshadri constants. In Section \ref{sc}, we study Seshadri constants
for ample line bundles $L$ on $\gr$.
We have two main results 
(Theorems \ref{sc1} and \ref{sc2})
on Seshadri constants. They assume
different conditions on the Harder-Narasimhan filtration of $E$. We
also prove two corollaries (Corollary \ref{cor-sc1} and Corollary
\ref{cor-sc2}) to our main theorems which give values of 
$\varepsilon(L,1)$ and $\varepsilon(L)$.
Our results generalize many of the known results in the case of ruled
surfaces (i.e., when $E$ has rank two); this is discussed in Remark
\ref{compare-rank2}. We also give several examples to
illustrate our results. 

\section{Unstable bundles and real N\'eron--Severi group}\label{grassmann}

\subsection{Nef and pseudo-effective cones of a Grassmann bundle}

Let $X$ be a connected smooth complex projective curve. The genus of $X$ will be denoted
by $g$. Let $E$ be a vector bundle over $X$ which is not semistable. Let
\begin{equation}\label{e1}
0\,=\, E_0\, \subset\, E_1\,\subset\, E_2\, \subset\, \cdots \, \subset\, E_{d-1}
\, \subset\, E_d \,=\, E
\end{equation}
be the Harder--Narasimhan filtration of $E$. Note that $d\, \geq\, 2$ because $E$ is
assumed not to be semistable.

Fix an integer $1\, <\, m\, \leq\, d$. Let
$$
r\,:=\, {\rm rank}(E/E_{m-1})\, .
$$
Let
\begin{equation}\label{e2}
f\,:\,\text{Gr}(E)\, \longrightarrow\, X
\end{equation}
be the Grassmann bundle that parametrizes the quotients of the fibers of $E$ of dimension $r$.
The real N\'eron--Severi group for $\text{Gr}(E)$ is defined to be
$$
\text{NS}(\text{Gr}(E))_{\mathbb R}\, :=\, 
\text{NS}(\text{Gr}(E))\otimes_{\mathbb Z}{\mathbb R}\,=\,
(\text{Pic}(\text{Gr}(E))/\text{Pic}^0(\text{Gr}(E)))\otimes_{\mathbb Z}{\mathbb R}\, ,
$$
where $\text{Pic}^0(\text{Gr}(E))$ is the connected component of
the Picard group $\text{Pic}(\text{Gr}(E))$
containing the identity element. A class $c\, \in\, \text{NS}(\text{Gr}(E))$ is called
effective if there is an effective divisor $D$ on $\text{Gr}(E)$ such that $c$ represents
${\mathcal O}_{\text{Gr}(E)}(D)$. The \textit{pseudo-effective cone} of $\text{Gr}(E)$ is
the closed cone in $\text{NS}(\text{Gr}(E))_{\mathbb R}$ generated by the effective classes.
The nef cone of $\text{Gr}(E)$ coincides with the closed cone in
$\text{NS}(\text{Gr}(E))_{\mathbb R}$
generated by the ample divisors on $\text{Gr}(E)$. The nef cone is contained in the
pseudo-effective cone.

For notational convenience, we will write
\begin{equation}\label{e3}
\theta:=\, \text{degree}(E/E_{m-1})\, .
\end{equation}

Let
\begin{equation}\label{cl}
{\mathcal L}\,\in\, \text{NS}(\text{Gr}(E))
\end{equation}
be the class of a fiber of the 
projection $f$ in \eqref{e2}, so $\mathcal L$ is the class of $f^*L_1$ for a line 
bundle $L_1$ on $X$ of degree one. The tautological line bundle on $\text{Gr}(E)$ 
(the $r$-th exterior product of the tautological quotient vector bundle) will be 
denoted by ${\mathcal O}_{\text{Gr}(E)}(1)$.

We will now recall Proposition 4.1 of page 363 in \cite{BP} since it will be 
used here a couple of times.
Take any integer $1\, \leq\, r''\, <\, \text{rank}(E)$. Let $1\,
\leq\, t\, \le \, d$ be the unique largest integer such that
$$
\text{rank}(E/E_{t-1}) \,\ge\, r''
$$
(see \eqref{e1}). 

For any vector bundle $V$ on $X$, denote ${\rm degree}(V)/{\rm rank}(V)\,\in\,
\mathbb Q$ by $\mu(V)$.

Define
\begin{equation}\label{et}
\theta_{E,r''}\, :=\, (r''-\text{rank}(E/E_{t}))\cdot \mu(E_{t}/E_{t-1})+\text{degree}
(E/E_{t})\, .
\end{equation}

Let $\varphi\, :\, {\rm Gr}_{r''}(E)\, \longrightarrow\, X$ be the Grassmann bundle
parametrizing the $r''$ dimensional quotients of the fibers of $E$. Let
${\mathcal O}_{{\rm Gr}_{r''}(E)}(1)\, \longrightarrow\, {\rm Gr}_{r''}(E)$ be the determinant
of the tautological bundle over ${\rm Gr}_{r''}(E)$ of rank $r''$. Just as in \eqref{cl}, let
$$
{\mathcal L}''\,\in\, \text{NS}(\text{Gr}_{r''}(E))
$$
be the class of a fiber of the above projection $\varphi$.

\begin{proposition}[{\cite[Proposition 4.1]{BP}}]\label{bp41}
The nef cone of ${\rm Gr}_{r''}(E)$ is generated by ${\mathcal L}''$ and
$[{\mathcal O}_{{\rm Gr}_{r''}(E)}(1)]- \theta_{E,r''}\cdot {\mathcal L}''$, where
${\mathcal L}''$ is defined above.
\end{proposition}

If there is a positive integer $c\, <\, d$ such that ${\rm rank}(E_c)\,=\, r$
(see \eqref{e1}), then define
\begin{equation}\label{dzeta}
\zeta\, :=\, {\rm degree}(E_c)\, .
\end{equation}

\begin{lemma}\label{lem1}
\mbox{}
\begin{enumerate}
\item The nef cone of ${\rm Gr}(E)$ is generated by $\mathcal L$ (defined in \eqref{cl}) and
${\mathcal M}\, :=\, [{\mathcal O}_{{\rm Gr}(E)}(1)]- \theta\cdot {\mathcal L}$ (see
\eqref{e3}).

\item Assume that there is a positive integer $c\, <\, d$ such that ${\rm rank}(E_c)\,=\, r$. Then
the pseudo-effective cone of ${\rm Gr}(E)$ is generated by $\mathcal L$ and
$[{\mathcal O}_{{\rm Gr}(E)}(1)]- \zeta\cdot {\mathcal L}$, where $\zeta$ is defined in
\eqref{dzeta}.
\end{enumerate}
\end{lemma}

\begin{proof}
The first statement is a particular case of Proposition \ref{bp41}.

The second statement is proved in \cite{BHP} (see \cite[p.~74, Theorem~4.1]{BHP}). Note that
the rational number $\zeta_{E,r}$, defined in \cite[p.~74, (3.10)]{BHP}, coincides with our $\zeta$.
\end{proof}

\begin{proposition}\label{prop1}
Assume that $\zeta$ in \eqref{dzeta} is an integral multiple of $r$.
Then the class $[{\mathcal O}_{{\rm Gr}(E)}(1)]- \zeta\cdot {\mathcal L}\,\in\,
{\rm NS}({\rm Gr}(E))$ is effective. Furthermore, there is exactly one effective divisor
on ${\rm Gr}(E)$ whose class is $[{\mathcal O}_{{\rm Gr}(E)}(1)]- \zeta\cdot {\mathcal L}$.
\end{proposition}

\begin{proof}
Let $L$ be a line bundle on $X$ such that $L^{\otimes r}\,=\, \bigwedge^r (E_c)^*$
(recall that $\zeta$ is an integral multiple of $r$). We note that $\text{degree}(E_c\otimes L)
\,=\,0$, in fact the determinant line bundle $\bigwedge^r (E_c\otimes L)$
is the trivial line bundle.

Replace $E$ by ${\mathcal E}\,:=\, E\otimes L$. Let
\begin{equation}\label{eh}
h\, :\, {\rm Gr}({\mathcal E})\, \longrightarrow\, X
\end{equation}
be the Grassmann bundle that parametrizes the quotients of the fibers of $\mathcal E$ of dimension $r$.
Then ${\rm Gr}({\mathcal E})$ is canonically identified with ${\rm Gr}(E)$.

Given a vector bundle $U'$ and a line bundle $L'$ on $X$, the vector bundle $U'\otimes L'$
is semistable if and only if $U'$ is semistable. Also, we have $\mu(U'\otimes L')\,=\,
\mu(U') +\mu(L')$. Therefore, the Harder--Narasimhan filtration of $U'\otimes L'$ is simply
the Harder--Narasimhan filtration of $U'$ tensored with the line bundle $L'$. In particular,
the Harder--Narasimhan filtration of $\mathcal E$ is the filtration in \eqref{e1} tensored
with $L$, meaning the filtration
$$
0\,=\, E_0\otimes L\, \subset\, E_1\otimes L\, \subset\, \cdots \, \subset\, E_{d-1}\otimes L
\, \subset\, E_d\otimes L \,=\, E\otimes L\,=\, {\mathcal E}
$$
is the Harder--Narasimhan filtration of $\mathcal E$.
Since $\text{degree}(E_c\otimes L)
\,=\,0$, the class of ${\mathcal O}_{{\rm 
Gr}({\mathcal E})}(1)$ in ${\rm NS}({\rm Gr}(E))\,=\, {\rm NS}({\rm Gr}( {\mathcal 
E}))$ coincides with the class $[{\mathcal O}_{{\rm Gr}(E)}(1)]- \zeta\cdot {\mathcal L}$ in 
Lemma \ref{lem1}(2).

We have
\begin{equation}\label{dtr}
\bigwedge\nolimits^r {\mathcal E}\, \supset\, \bigwedge\nolimits^r (E_c\otimes L)\,=\,
(\bigwedge\nolimits^r E_c)\otimes L^{\otimes r}\,=\, {\mathcal O}_X\, .
\end{equation}
This inclusion of the coherent sheaf
${\mathcal O}_X$ in $\bigwedge\nolimits^r {\mathcal E}$ produces an inclusion
$$
{\mathbb C}\,=\, H^0(X,\, {\mathcal O}_X)
\, \hookrightarrow\, H^0(X,\, \bigwedge\nolimits^r {\mathcal E})\,=\,
H^0({\rm Gr}({\mathcal E}),\, {\mathcal O}_{{\rm Gr}({\mathcal E})}(1))\, .
$$
Let $D$ be the effective divisor on ${\rm Gr}(E)\,=\, {\rm Gr}({\mathcal E})$ given by
this one-parameter family of nonzero sections.

Since the class of ${\mathcal O}_{{\rm Gr}({\mathcal E})}(1)$ in ${\rm NS}({\rm Gr}(E))$ is 
$[{\mathcal O}_{{\rm Gr}(E)}(1)]- \zeta\cdot {\mathcal L}$, the effective divisor $D$ 
constructed above lies in the class $[{\mathcal O}_{{\rm Gr}(E)}(1)]- \zeta\cdot {\mathcal L}$.
In particular, the class $[{\mathcal O}_{{\rm Gr}(E)}(1)]- \zeta\cdot {\mathcal L}$ is effective.

Now we will show that there is only one effective divisor in the class $[{\mathcal O}_{{\rm 
Gr}(E)}(1)]- \zeta\cdot {\mathcal L}\,=\, [{\mathcal O}_{{\rm Gr}({\mathcal E})}(1)]$.

Let $D'$ be an effective divisor on ${\rm Gr}({\mathcal E})$ that lies in the
class $[{\mathcal O}_{{\rm Gr}({\mathcal E})}(1)]$. Then
\begin{equation}\label{ed}
{\mathcal O}_{{\rm Gr}({\mathcal E})}(D')\,=\, {\mathcal O}_{{\rm Gr}({\mathcal E})}(1)\otimes
h^*L_0\, ,
\end{equation}
where $L_0$ is a line bundle on $X$ of degree zero and $h$ is the projection
to $X$ in \eqref{eh}. Now using the projection formula, we have
$$
H^0({\rm Gr}({\mathcal E}),\, {\mathcal O}_{{\rm Gr}({\mathcal E})}(1)\otimes
h^*L_0)\,=\, H^0(X,\, (\bigwedge\nolimits^r {\mathcal E})\otimes L_0)\, .
$$

Let $U$ be a vector bundle on $X$, let $U_1\, \subset\, U$ be the first nonzero term in the
Harder--Narasimhan filtration for $U$. For any integer $1\, \leq\, s\, \leq\, \text{rank}(U_1)$,
the first nonzero term in the Harder--Narasimhan filtration for the vector bundle $\bigwedge^s U$
is $\bigwedge^s U_1$. Furthermore, if $U'$ is any nonzero term in the
Harder--Narasimhan filtration for $U$, then the first nonzero term
in the Harder--Narasimhan filtration for $\bigwedge^b U$, where
$b\,=\, \text{rank}(U')$, is $\bigwedge^b U'$.

Since $r\, =\, \text{rank}(E_c\otimes L)$, from the above property of the Harder--Narasimhan
filtrations we know that the first nonzero term in the Harder--Narasimhan filtration for
$\bigwedge\nolimits^r {\mathcal E}$ is the line bundle 
$\bigwedge\nolimits^r (E_c\otimes L)\,=\, (\bigwedge\nolimits^r E_c)\otimes L^{\otimes r}
\,=\, {\mathcal O}_X$ (as in \eqref{dtr}). Therefore, the first nonzero term in the
Harder--Narasimhan filtration for
$(\bigwedge\nolimits^r {\mathcal E})\otimes L_0$ is the line bundle
$\bigwedge\nolimits^r (E_c\otimes L)\otimes L_0\,=\, L_0$ (recall that the Harder--Narasimhan
filtration of the tensor product of $U$ with a line bundle is the tensor product of the
Harder--Narasimhan filtration $U$ with the line bundle).
We have $\text{degree}(L_0)\,=\, 0$, so all the other graded pieces of the
Harder--Narasimhan filtration for $(\bigwedge\nolimits^r {\mathcal E})\otimes L_0$ (other
than $L_0$) are of negative degree. Since a semistable vector bundle of negative degree
does not admit any nonzero section, from the Harder--Narasimhan filtration for
$(\bigwedge\nolimits^r {\mathcal E})\otimes L_0$ it follows that
\begin{equation}\label{ed2}
H^0(X,\, (\bigwedge\nolimits^r {\mathcal E})\otimes L_0)\, \subset\,
H^0(X,\, L_0)\, .
\end{equation}
On the other hand, from \eqref{ed} it follows that
$$
H^0(X,\, (\bigwedge\nolimits^r {\mathcal E})\otimes L_0)\, \not=\, 0\, .
$$
Since $\text{degree}(L_0)\,=\, 0$, from \eqref{ed2} it now follows that
\begin{itemize}
\item $L_0$ is the trivial line bundle ${\mathcal O}_X$, and

\item $\dim H^0(X,\, (\bigwedge\nolimits^r {\mathcal E})\otimes L_0)\, =\, 1$.
\end{itemize}
This completes the proof of the proposition.
\end{proof}

\subsection{Bound on the degree of a quotient}

Recall that $m$ is a fixed integer with $1\,<\,m\,\le\, d$, and $r\,=\,
{\rm rank}(E/E_{m-1})$. The following is a corollary of Lemma \ref{lem1}(1):

\begin{corollary}\label{cor1}
Let $W$ be a quotient bundle of $E$ of rank $r$. Then
${\rm degree}(W) \, \geq\, \theta$.
\end{corollary}

\begin{proof}
Let
$$
s_W\, :\, X\, \longrightarrow\, {\rm Gr}(E)
$$
be the section of the projection $f$ in \eqref{e2} defined by $W$. For
the class $\mathcal M$ in Lemma \ref{lem1}(1), we have
$$
s^*_W({\mathcal M})\,=\, s^*_W([{\mathcal O}_{{\rm Gr}(E)}(1)])- \theta\, \geq\, 0
$$
($\text{NS}(X)$ is identified with $\mathbb Z$ using degree),
because $\mathcal M$ is nef by Lemma \ref{lem1}(1) and $s^*_W({\mathcal L})\,=\, 1$.
But $$s^*_W([{\mathcal O}_{{\rm Gr}(E)}(1)])\,=\, \text{degree}(W)\, ,$$ hence we have
$\text{degree}(W) \, \geq\, \theta$.
\end{proof}

The following is a refinement of Corollary \ref{cor1}:

\begin{proposition}\label{lem2}
Let $W$ be a quotient bundle of $E$ of rank $r$. Then exactly one of the following two
is valid:
\begin{enumerate}
\item The quotient $W$ coincides with the quotient $E/E_{m-1}$ in \eqref{e1}.

\item ${\rm degree}(W) \, \geq\, \theta + \mu(E_{m-1}/E_{m-2})- \mu(E_m/E_{m-1})$.
\end{enumerate}
\end{proposition}

\begin{proof}
Since $ \mu(E_m/E_{m-1})\, <\, \mu(E_{m-1}/E_{m-2})$, if statement (2) in the proposition
holds, then statement (1) is not valid. So it suffices to prove that (2) holds assuming that
(1) does not hold.

Consider the quotient bundles $W$ and $E/E_{m-1}$ of $E$ of rank $r$. Let $S$
be the image in $W$ of the subbundle $E_{m-1}\, \subset\, E$. Therefore,
$$
Q\, :=\, W/S
$$
is a quotient of $E/E_{m-1}$. So, we have
\begin{equation}\label{f1}
\text{degree}(W)\,=\, \text{degree}(S)+\text{degree}(Q)\, .
\end{equation}
Assume that $W$ is not the quotient $E/E_{m-1}$ of $E$. This implies that we have
$S\, \not=\, 0$ and $\text{rank}(Q)\, <\, r$. Let $r'$ be the rank of $Q$; so
${\rm rank}(S)\,=\, r-r'\, >\, 0$.

Consider the Harder--Narasimhan filtration of $E$ in \eqref{e1}.
Let $t$ be the smallest positive integer such that $\text{rank}(E/E_{t})
\, <\, r'$. Note that $t\, > \, m-1$ because $r'\, <\, r$.
Let
$$
q_0\, :\, \text{Gr}_{r'}(E/E_{m-1})\,\longrightarrow\, X
$$
be the Grassmann bundle that parametrizes all the $r'$ dimensional quotients of the fibers
of $E/E_{m-1}$. The real N\'eron--Severi class $$[{\mathcal O}_{\text{Gr}_{r'}(E/E_{m-1})}(1)]
-(\text{degree}(E/E_t)+ (r'-\text{rank}(E/E_{t}))\cdot
\mu (E_t/E_{t-1}))\cdot q^*_0 L_1$$
on $\text{Gr}_{r'}(E/E_{m-1})$ is nef (see Proposition \ref{bp41}), where
${\mathcal O}_{\text{Gr}_{r'}(E/E_{m-1})}(1)$ is the tautological line bundle, and
$L_1$ is the class of a line bundle on $X$ of degree $1$. Consequently, evaluating this class on
the section of the projection $q_0$ given by the torsionfree quotient of $Q$ we have
\begin{equation}\label{f2}
\text{degree}(Q) \, \geq\, \text{degree}(Q/{\rm Torsion})
\, \geq\, \text{degree}(E/E_t)+ (r'-\text{rank}(E/E_{t}))\cdot
\mu (E_t/E_{t-1})\, .
\end{equation}

We will next prove that
\begin{equation}\label{f3}
\text{degree}(S)\, \geq\, (r-r')\cdot\mu(E_{m-1}/E_{m-2})\, .
\end{equation}

For this, first note that if the natural surjective homomorphism
$E_{m-1}\, \longrightarrow\, S$ is also injective, then \eqref{f3} holds,
because in that case,
$$
\text{degree}(E_{m-1})\, \geq\, \text{rank}(E_{m-1})\cdot \mu(E_{m-1}/E_{m-2})
\,=\, (r-r')\cdot\mu(E_{m-1}/E_{m-2})\, .
$$
So to prove \eqref{f3} we will assume that $r-r'\, <\, \text{rank}(E_{m-1})$.

Let $u\,\leq\, m-1$ be the smallest nonnegative integer such that
$0\,\leq\,\text{rank}(E_{m-1})-\text{rank}(E_u)\, < \, r-r'$; note that
$u\, \geq\, 1$, because $r-r'\, <\, \text{rank}(E_{m-1})$. Let
$$
q_1\, :\, \text{Gr}_{r-r'}(E_{m-1})\,\longrightarrow\, X
$$
be the Grassmann bundle that parametrizes all the $r-r'$ dimensional quotients
of the fibers of $E_{m-1}$. 

By Proposition \ref{bp41},  the line bundle 
$$
[{\mathcal O}_{\text{Gr}_{r-r'}(E_{m-1})}(1)]- (\text{degree}(E_{m-1}/E_{u}) +
(r-r'-\text{rank}(E_{m-1})+\text{rank}(E_u)) \cdot\mu(E_u/E_{u-1}))\cdot q^*_1 L_1
$$
on $\text{Gr}_{r-r'}(E_{m-1})$ is nef. 
Evaluating this line bundle on the
section of $q_1$ given by $S$ (note that $S$ is a locally free quotient of
$E_{m-1}$ of rank $r-r'$), it follows that
\begin{eqnarray}
\text{degree}(S)\, &\geq&\, \text{degree}(E_{m-1}/E_{u}) +
(r-r'-\text{rank}(E_{m-1})+\text{rank}(E_u))\cdot\mu(E_u/E_{u-1}) \nonumber\\
&=&\, \text{rank}(E_{m-1}/E_{u})\cdot\mu(E_{m-1}/E_{u}) \label{new} \\
&& +~
(r-r'-\text{rank}(E_{m-1})+\text{rank}(E_u))\cdot\mu(E_u/E_{u-1}) \nonumber\\
\, &\geq&\, (\text{rank}(E_{m-1}) - \text{rank}(E_u))\cdot
          \mu(E_{m-1}/E_{m-2}) \label{new1} \\
&&+~ (r-r'-\text{rank}(E_{m-1})
   +\text{rank}(E_u))\cdot\mu(E_{m-1}/E_{m-2}). \nonumber\\
&  = & (r-r')\cdot\mu(E_{m-1}/E_{m-2}) \nonumber.
\end{eqnarray}

The inequality \eqref{new1} above holds because $$\mu(E_{m-1}/E_{m-2})\, \leq\,
\mu(E_{m-1}/E_{u}) \text{~and~}
\mu(E_{m-1}/E_{m-2})\, \leq\, \mu(E_u/E_{u-1}).$$ 

Note that the first
inequality above holds if $u < m-1$. If $u=m-1$ then we directly get
\eqref{f3} from \eqref{new}. 

This completes the proof of the inequality \eqref{f3}. 


{}From \eqref{f1}, \eqref{f2} and \eqref{f3} we have
$$
\text{degree}(W)\, \geq\, \text{degree}(E/E_t)+ (r'-\text{rank}(E/E_{t}))\cdot
\mu (E_t/E_{t-1})+ (r-r')\cdot\mu(E_{m-1}/E_{m-2})
$$
$$
=\, \text{degree}(E/E_t)+ (r'-\text{rank}(E/E_{t}))\cdot
\mu (E_t/E_{t-1})+ (r-r')\cdot\mu(E_{m}/E_{m-1})
$$
\begin{equation}\label{z2}
+
(r-r')\cdot(\mu(E_{m-1}/E_{m-2})-\mu(E_{m}/E_{m-1}))\, .
\end{equation}

We will show that
\begin{equation}\label{z1}
\theta\, \leq\, \text{degree}(E/E_t)+
(r'-\text{rank}(E/E_{t}))\cdot
\mu (E_t/E_{t-1})+ (r-r')\cdot\mu(E_{m}/E_{m-1})\, .
\end{equation}

If $t\,=\,m$, then the two sides of \eqref{z1} coincide, because
$(r'-\text{rank}(E/E_{t}))+(r-r')\,=\, \text{rank}(E_{m}/E_{m-1})$. If
$t\,> \,m$, then
$$ 
\theta\,=\,
\text{degree}(E/E_t)+\text{degree}(E_t/E_{t-1})+\text{degree}(E_{t-1}/E_{m-1})
$$
$$
=\,\text{degree}(E/E_t)+(r'-\text{rank}(E/E_t))\mu(E_t/E_{t-1})+(r-r')
\cdot\mu(E_{t-1}/E_{m-1})
$$
$$
+(\text{rank}(E/E_{t-1})-r')\mu(E_t/E_{t-1})+(r'-r+\text{rank}(E_{t-1}/E_{m-1}))
\mu(E_{t-1}/E_{m-1})\, ;
$$
also, $(\text{rank}(E/E_{t-1})-r')\,=\, -(r'-r+\text{rank}(E_{t-1}/E_{m-1}))\,
>\, 0$, and $\mu(E_t/E_{t-1})\,<\, \mu(E_{t-1}/E_{m-1})$, so
$$
(\text{rank}(E/E_{t-1})-r')\mu(E_t/E_{t-1})+(r'-r+\text{rank}(E_{t-1}/E_{m-1}))
\mu(E_{t-1}/E_{m-1})\, <\, 0\, ,
$$
proving \eqref{z1}.

Combining \eqref{z2} and \eqref{z1},
$$
\text{degree}(W)\, \geq\, \theta +(r-r')\cdot(\mu(E_{m-1}/E_{m-2})-\mu(E_{m}/E_{m-1}))
$$
$$
\geq\, \theta +\mu(E_{m-1}/E_{m-2})-\mu(E_{m}/E_{m-1})
$$
because $r-r'\, \geq\, 1$ and $\mu(E_{m-1}/E_{m-2})\,>\, \mu(E_{m}/E_{m-1})$.
This completes the proof.
\end{proof}

Write $\theta$ in \eqref{e3} as
$$
\theta\,=\, s_0 + \alpha \cdot r\, ,
$$
where $\alpha\, \in\, \mathbb Z$ and $-r\, \leq\, s_0\,<\, 0$. Let $L$
be a line bundle
on $X$ of degree $-\alpha$ and consider $F\,:=\, E\otimes L$. Note that $\theta_{F,r}\,=\, 
s_0$ (see \eqref{et}).

As mentioned before, given a vector bundle $U'$ and a line bundle $L'$ on $X$, the
Harder--Narasimhan filtration of $U'\otimes L'$ is
the Harder--Narasimhan filtration of $U'$ tensored with the line bundle $L'$. In particular,
the Harder--Narasimhan filtration for $F$ is the filtration in \eqref{e1} tensored with $L$.

As noted before, the Grassmann bundles for $E$ and $F$ are identified.

Therefore, substituting $E$ by $E\otimes L$ we can assume that $-r\,\leq\, \theta\, <\, 0$. 
Henceforth without loss of generality it may be assumed that 
\begin{equation}\label{theta} 
-r\,\leq\, \theta\, <\, 0\, .
\end{equation}

Let $Y$ be an irreducible smooth complex projective curve and
$$
\phi\, :\, Y\, \longrightarrow\, X
$$
a surjective morphism.

\begin{corollary}\label{cor2}
Assume that
\begin{equation}\label{as}
\mu(E_m/E_{m-1})-\mu(E_{m-1}/E_{m-2})\, \leq\, \theta\, <\, 0\, .
\end{equation}
Take any quotient
bundle $Q\, \longrightarrow\, Y$ of rank $r$ of $\phi^*E$. Then exactly one of the following 
is valid:
\begin{enumerate}
\item The quotient $Q$ coincides with the quotient $\phi^*(E/E_{m-1})\,=\, (\phi^*E)/(\phi^*E_{m-1})$.

\item ${\rm degree}(Q) \, \geq\, 0$.
\end{enumerate}
\end{corollary}

\begin{proof}
If $U$ is a semistable vector bundle on $X$, then the vector bundle $\phi^*U$ on $Y$ is
also semistable \cite[p.~441, Theorem 2.4]{BS}. (Note that this statement
for vector bundles on curves is much easier to prove than the version in \cite{BS}
which is for all dimensions.) Also, we have
$$\mu(\phi^*U) \, =\, \text{degree}(\phi)\cdot \mu^*(U)\, .$$ From these two facts it follows
immediately that the Harder--Narasimhan filtration of $\phi^*U$ is the pullback, via $\phi$, of
the Harder--Narasimhan filtration of $U$. In particular, the
Harder--Narasimhan filtration of $\phi^*E$ is the pullback, via $\phi$, of the filtration
in \eqref{e1}. Therefore, the result follows from Proposition \ref{lem2} and the
conditions on $\theta$ in \eqref{as}.
\end{proof}

\section{Seshadri constants}\label{sc}

In this section, we will use the results in Section \ref{grassmann} to
compute Seshadri constants of ample line bundles on $\text{Gr}(E)$. 

We quickly recall our set-up. Let $X$ be a smooth complex projective curve, and let 
$E$ be a vector bundle on $X$ of rank $n$ which is not semistable. We have its
Harder-Narasimhan filtration in \eqref{e1}. As in Section \ref{grassmann}, we fix an 
integer $1 \,<\, m \,\le\, d$ and set $r \,:=\, {\rm rank}(E/E_{m-1})$.
Let $f\, :\, \text{Gr}(E)\,\longrightarrow\, 
X$ be the Grassmann bundle that parametrizes the quotients of the fibers of $E$ of 
dimension $r$.

In order to compute the Seshadri constants of ample line bundles on
$\gr$, we will first describe its closed cone of curves ${\rm \overline{NE}}(\gr)$ (the dual
of the nef cone of $\gr$).

Let $\Gamma_l$ denote the class of a line in a fiber $Y\, :=\, f^{-1}(x_0)$.
Note that $Y$ is isomorphic to the Grassmannian variety of
$r$-dimensional quotients of an $n$-dimensional vector space. 
Let $\Gamma_s$ denote the image of the section 
$$s\,:\, X \,\longrightarrow\, \gr$$ corresponding to the rank $r$ quotient $E
\,\longrightarrow\, E/E_{m-1}$. 

We claim that ${\rm \overline{NE}}(\gr)$ is the cone spanned by the
classes of $\Gamma_s$ and $\Gamma_l$. This can be deduced from the facts that
${\rm \overline{NE}}(\gr)$ is the dual cone to the nef cone of $\gr$,
and the nef cone of $\gr$ is generated by $\mathcal{L}$ and $\mathcal{M}$
by Lemma \ref{lem1}(1). Indeed, it is easy to see that $\mathcal{L}\cdot
\Gamma_l \,=\, 0$ and $\mathcal{L}\cdot \Gamma_s \,=\, 1$. On the other hand,
$\str_{\gr}(1) \cdot \Gamma_s \,=\, \theta$, since by definition, 
$\theta \,=\,{\rm degree}(E/E_{m-1})$ (see \eqref{e3}). It is also clear
that $\str_{\gr}(1) \cdot \Gamma_l\,=\,1$. Now recall that $\mathcal{M} \,=\,
[\str_{\gr}(1)]-\theta \mathcal{L}$. This gives that $\mathcal{M} \cdot
\Gamma_s \,=\, 0$ and $\mathcal{M} \cdot
\Gamma_l \,= \,1$. It now follows that ${\rm \overline{NE}}(\gr)$ is the cone
spanned by the classes of $\Gamma_s$ and $\Gamma_l$, proving the claim.

We will now prove two independent results on Seshadri constants of
ample line bundles on $\gr$. These results make different
assumptions on the Harder-Narasimhan filtration \eqref{e1} of $E$.

\begin{theorem}\label{sc1}
Assume that there is an
integer $c\, <\, d$ such that ${\rm rank}(E_c)\,=\, r$
and that 
${\rm degree}(E_c)$ is an integral multiple of $r$. 
Let $L$ be an ample line bundle on ${\rm Gr}(E)$ which is 
numerically equivalent to $\mathcal{L}^{\otimes a} \otimes
\mathcal{M}^{\otimes b}$, 
where $\mathcal{L}$ and
$\mathcal{M}$ are as in Lemma \ref{lem1}(1). Then the Seshadri
constants of $L$ are given by the following: 
\begin{enumerate}
\item $\varepsilon(L,x) \,\ge\, {\rm min}\{a,\,b\}$, for all $x \,\in\,
{\rm Gr}(E)$. 

\item If $b \,\le\, a$, then $\varepsilon(L,x)\,=\, b$, for all $x\,\in\,
{\rm Gr}(E)$.

\item If $a \,<\, b$, then:
\begin{enumerate}
\item[(i)] if $x$ does not belong to the base locus of the linear system
$|\str_{\gr}(1)|$, then $\varepsilon(L,x) \,=\, b$; 

\item[(ii)] if $x$ belongs to the base locus of $|\str_{\gr}(1)|$, 
then $a\,\le\,\varepsilon(L,x) \,\le\, b$; 

\item[(iii)] if $x \,\in\, \Gamma_s$, then $\varepsilon(L,x)\,=\,a$. 
\end{enumerate}
\end{enumerate}
\end{theorem}

\begin{proof}
We first prove statements (1) and (2), which, in fact, hold in general, without
the assumption on $E_c$. 

Let $C\,\subset\,\gr$ be an irreducible and reduced curve passing
through $x$. Let $m$ denote the multiplicity of $C$ at $x$. We
consider two cases. 

{\bf \underline{Case 1}}: $C$ is contained in a fiber of $f$. 
\smallskip

In this case, as an element of ${\rm \overline{NE}}(\gr)$, this $C$ is
given by $n_l\cdot\Gamma_l$ for a positive integer
$n_l$. In other words, $C$ is a curve of degree $n_l$ in a fiber of
$f$ which is isomorphic to the Grassmannian of $r$-dimensional
quotients of an $n$-dimensional vector space. Since $C$ contains a
point $x$ of multiplicity $m$, its degree is at least $m$, i.e., $n_l
\,\ge\, m$.

We then have $\frac{L\cdot C}{m}\,=\, \frac{bn_l}{m}\,\ge\, b$. 

\bigskip

{\bf \underline{Case 2}}: $C$ is not contained in a fiber of $f$. 
\smallskip

In this case, as an element of ${\rm \overline{NE}}(\gr)$, this $C$ is
given by $n_l\Gamma_l + n_s \Gamma_s$, for non-negative
integers $n_s, n_l$. Let $Y$ be the fiber of $f$
passing through $x$. Since $C \,\not\subset\, Y$, B\'ezout's theorem
gives us the inequality $$Y\cdot C \,=\, \mathcal{L}\cdot C \,\ge\,
\text{mult}_x C \,=\, m\, .$$ So we have $n_s \,\ge\, m$.

Now the Seshadri ratio is given by $\frac{L\cdot C}{m} \,=\,
\frac{an_s+bn_l}{m} \,\ge\, a + \frac{bn_l}{m} \,\ge\, a$. 

By Cases 1 and 2, we have that $\frac{L\cdot C}{m}\,\ge\, 
\text{min}\{a,\, b\}$, for all $x \,\in\, \gr$. This proves (1). Further, note 
that for every point $x \,\in\, \gr$, there exists a line $l$
in a fiber of $f$ that contains $x$. The Seshadri ratio for
$l$ is $\frac{L\cdot l}{1} \,=\, b$. This means that $\varepsilon(L,x)\,\le\,
b$ for all $x \,\in\, \gr$. So if $b \,\le\, a$, then 
$\varepsilon(L,x)\,=\,b$,
giving (2). 

Now we use the hypothesis of the theorem to prove (3). Because of our hypothesis, 
we can apply Proposition \ref{prop1}. 
We normalize $E$ as in the proof 
of Proposition \ref{prop1}. 
In other words, we replace $E$ by $E \otimes L$, where $L$ is a
line bundle on $X$ such that $L^{\otimes r}\,=\, \bigwedge^r (E_c)^*$
(recall that $\text{degree}(E_c)$ is a multiple of $r$, by hypothesis). Then, as argued in the
proof of Proposition \ref{prop1}, it follows that 
$\bigwedge\nolimits^r E_c \,=\, \str_X$. So $\zeta :=
\text{degree}(E_c) \,=\, 0$ and $\theta \,=\, {\rm degree}(E/E_{m-1}) \,<\, 0$. 
Further, by Proposition \ref{prop1}, 
$\str_{\gr}(1)$ is
effective. 

Now we prove (3)(i). So suppose that $x$ is not in the base 
locus of the linear system
$|\str_{\gr}(1)|$. Then $C$ is also not contained in the base locus of 
$|\str_{\gr}(1)|$. This implies that $\str_{\gr}(1) \cdot C\,\ge\,
0$. 
Suppose that as an element of ${\rm \overline{NE}}(\gr)$, the curve $C$ is
given by $n_l\Gamma_l + n_s \Gamma_s$, where $n_s,\, n_l$ are
non-negative integers. Then from the observation that
$\str_{\gr}(1) \cdot C\,\ge\, 0$ it follows that $n_l \,\ge\, n_s(-\theta)\,\ge\, n_s$. 

If $C$ is contained in a fiber of $f$, then by Case 1, we have
$\frac{L\cdot C}{m}\,\ge\, b$. If $C$ is not contained in a fiber of 
$f$, then we have $n_s \,\ge\, m$. So $n_l\,\ge\, n_s \,\ge\, m$. Hence 
$\frac{L\cdot C}{m} \,=\,
\frac{an_s+bn_l}{m} \,\ge\, a + b \ge b$. This completes the proof of
(3)(i).

For (3)(ii), we have $\varepsilon(L,x) \ge a$, by part (1) of the
theorem. Also, as noted above, for every point
$x \,\in\, \gr$, there exists a line $l$
in a fiber of $f$ that contains $x$ whose Seshadri ratio is given by
$\frac{L\cdot l}{1} \,=\, b$. This means that $\varepsilon(L,x)\,\le\,
b$ for all $x \,\in\, \gr$. So (3)(ii) follows.

For (3)(iii),
note that $\Gamma_s$ is smooth and $L\cdot \Gamma_s \,=\, a$. 
\end{proof}

\begin{corollary}\label{cor-sc1}
Assume the hypotheses of Theorem \ref{sc1} hold. 
Let $L \,=\, \mathcal{L}^{\otimes a} \otimes \mathcal{M}^{\otimes b}$ be
an ample line bundle on ${\rm Gr}(E)$, where $\mathcal{L}$ and
$\mathcal{M}$ are as in Lemma \ref{lem1}(1). Then we have 
\begin{enumerate}
\item $\varepsilon(L,1) \,= \, b$, and
\item $\varepsilon(L) \,= \,{\rm min}\{a, \, b\}$.
\end{enumerate}
\end{corollary}
\begin{proof}
Recall that $\varepsilon(L,1)$ is the supremum of $\varepsilon(L,x)$ as
$x$ varies in $X$, while $\varepsilon(L)$ is the infimum of 
$\varepsilon(L,x)$ as $x$ varies in $X$ (see Section \ref{se1}).

If $b \le a$, then Theorem \ref{sc1} implies that 
$\varepsilon(L,x) = b$ for all $x$. 
So we have $\varepsilon(L,1) = \varepsilon(L) = b$. 
On the other hand, if $b > a$, then $\varepsilon(L,x)= b$ for a point
$x$ outside the base locus of $|\str_{\gr}(1)|$
and $a \le \varepsilon(L,x) \le b$ if $x$ is in the base locus of 
$|\str_{\gr}(1)|$. Note also that 
$\varepsilon(L,x)= a$ if $x \in \Gamma_s$. So the corollary
follows. 
\end{proof}

For our next result about Seshadri constants, we will make use of
Corollary \ref{cor2}. We also normalize $E$ as in \eqref{theta}. 

\begin{theorem}\label{sc2}
Assume that $\mu(E_m/E_{m-1})-\mu(E_{m-1}/E_{m-2})\, \leq\, \theta$.
Let $L$ be an ample line bundle on ${\rm Gr}(E)$ which is 
numerically equivalent to $\mathcal{L}^{\otimes a} \otimes
\mathcal{M}^{\otimes b}$, 
where $\mathcal{L}$ and
$\mathcal{M}$ are as in Lemma \ref{lem1}(1). 
The Seshadri
constant of $L$ at $x\, \in\, {\rm Gr}(E)$ is given by the following: 
\begin{equation*}
\varepsilon(L,x) \,=\, 
\begin{cases}
b & \text{~if~} b \le a \text{~or~} x \notin \Gamma_s, \\
a & \text{~if~} a < b \text{~and~} x \in \Gamma_s\, .
\end{cases}
\end{equation*}
\end{theorem}

\begin{proof}
We normalize $E$ as in \eqref{theta}. Then 
$\theta \,=\, {\rm degree}(E/E_{m-1})$ is negative.

Note that, if $C$ is an irreducible curve in $\gr$, not contained in a fiber of $f$, then the restriction of $f$ to $C$ 
defines a surjective morphism $\phi:Y\to X$, where $Y$ is the normalization of $C$. Moreover, the embedding of $C$ in $\gr$ defines a 
rank $r$ quotient $Q$ of $\phi^{\star}(E)$. Conversely, for any surjective morphism $\phi: Y\to X$ with $Y$ smooth and irreducible, a rank $r$ quotient of $\phi^*(E)$ defines a morphism $Y\to \gr$ and hence a curve $C$ in $\gr$, not contained in a fiber of $f$. In particular, the section
$\Gamma_s$, by definition, is given by the rank $r$ quotient 
 $E \,\longrightarrow\, E/E_{m-1}$.

Since $\mu(E_m/E_{m-1})-\mu(E_{m-1}/E_{m-2})\, \leq\, \theta$ by
hypothesis, we can apply Corollary \ref{cor2}. This corollary says
that if $C$ is a curve on $\gr$ given by a rank $r$ quotient $\phi^*(E) \to
Q$, then either $Q$ coincides with the quotient 
$E \,\longrightarrow\, E/E_{m-1}$ or $\text{degree}(Q) \ge 0$. In
other words, either $C = \Gamma_s$ or $\str_{\gr}(1) \cdot C \,\ge \, 0$.

Let $C \,\subset\, \gr$ be an irreducible and reduced curve passing
through $x$. Let $m$ denote the multiplicity of $C$ at $x$. We
consider three different cases. 

{\bf \underline{Case 1}}: $C$ is contained in a fiber of $f$. 
\smallskip

In this case, we have $\frac{L\cdot C}{m}\, \ge\, b$. 
The proof is exactly the same as in Case 1 in the proof of Theorem \ref{sc1}.

\bigskip

{\bf \underline{Case 2}}: $C\,=\,\Gamma_s$.
\smallskip

Note that $\Gamma_s$ is a smooth curve, so that $m\,=\,1$. Hence the Seshadri
ratio is given by $\frac{L\cdot \Gamma_s}{1} \,=\, a$. 

\bigskip

{\bf \underline{Case 3}}: $C \,\ne\, \Gamma_s$ and $C$ is not contained in a fiber of 
$f$.
\smallskip

By Corollary \ref{cor2}, we have $\str_{\gr}(1) \cdot C \,\ge\, 0$. 

As before, as an element of ${\rm \overline{NE}}(\gr)$, the curve $C$ is
given by $n_l\Gamma_l + n_s \Gamma_s$, for non-negative
integers $n_s, n_l$. The same argument as for Case 2 in the
proof of Theorem \ref{sc1} now gives $n_l \,\ge\,
n_s(-\theta)\,\ge\, n_s \ge m$, and hence $\frac{L\cdot C}{m}\,=\,
\frac{an_s+bn_l}{m} \,\ge\, a+b \,\ge\, b$. 

The above analysis shows that if $C \,\ne\, \Gamma_s$, then the Seshadri
ratio $\frac{L\cdot C}{\text{mult}_x C}$ is at least $b$. Note that
given any point $x \,\in\, \gr$, there exists a line $l$ passing through $x$
lying in the fiber containing $x$. Since $l$ is smooth and it is given by
$\Gamma_l$ in the cone of curves of $\gr$, the Seshadri ratio
is $\frac{L\cdot l}{1}\,=\, L\cdot \Gamma_l \,=\, b$. 
Thus the Seshadri constant of $L$ at $x$
is equal to $b$, except when $a \,<\, b$ and $x \,\in\, \Gamma_s$, in which
case $\varepsilon(L,x) \,= \,a$. 

This completes the proof of the theorem.
\end{proof}

\begin{corollary}\label{cor-sc2}
Assume the hypotheses of Theorem \ref{sc2} hold. 
Let $L \,=\, \mathcal{L}^{\otimes a} \otimes \mathcal{M}^{\otimes b}$ be
an ample line bundle on ${\rm Gr}(E)$, where $\mathcal{L}$ and
$\mathcal{M}$ are as in Lemma \ref{lem1}(1). Then we have 
\begin{enumerate}
\item $\varepsilon(L,1) \,= \, b$, and
\item $\varepsilon(L) \,= \,{\rm min}\{a, \, b\}$.
\end{enumerate}
\end{corollary}
\begin{proof}
The proof is very similar to the proof of Corollary \ref{cor-sc1} and
follows easily from Theorem \ref{sc2}.
\end{proof}

\section{Remarks and Examples} 

\begin{remark}\label{compare-rank2}
As mentioned in Section \ref{se1}, several authors have studied Seshadri constants on 
$\gr$ when $E$ has rank 2 (see \cite{F,S}). In this remark, we compare our result in
Theorem \ref{sc2} with known results in this case. 

Let $X$ be a smooth complex projective curve and let $E$ be a rank 2 vector bundle 
on $X$. Assume that $E$ is normalized in the sense of \cite[Chapter V, Section 
2]{Har-AG}. In other words, $H^0(X,E) \,\ne\, 0$, but $H^0(X,\, E\otimes L)\,=\, 
0$ for any line bundle $L$ on $X$ with ${\rm degree}(L) \,<\, 0$.

In \cite[Chapter V, Section 2]{Har-AG}, the invariant 
$e$ of $\gr$ is defined as $e: \,=\, -\text{degree}(E)$, after normalizing $E$ as above. 
Very precise results on Seshadri constants are 
known when $e \,>\, 0$ (see \cite{F,S}); Theorem \ref{sc2} recovers these 
results.

We first claim that $E$ is not semistable if and only if $e > 0$
which in turn is equivalent to ${\rm degree}(E)
\, <\, 0$, by definition of $e$. 

Indeed, if $H^0(X,\,E) \,\ne \,0$, then $\str_X$ is a
subsheaf of $E$ of slope zero, so if $\text{degree}(E) \,<\, 0$, then $E$
can not be semistable.
On the other hand, if $E$ is not semistable,
consider the Harder-Narasimhan filtration \eqref{e1} of $E$: 
$$0\,=\, E_0\, \subset\, E_1\,\subset\, E_2 \,=\, E\, ,$$ where $E_1$ is a
sub-line-bundle of $E$ with maximal degree. Then ${\rm degree}(E_1) \, =\, 0$. For,
since $E$ has nonzero sections, ${\rm degree}(E_1) \, \ge\, 0$. But 
as $E$ has no nonzero sections after tensoring with any negative degree line
bundle, degree of $E_1$ has to be zero. Now it follows that ${\rm degree}(E) \,<\,
0$, since $E_1$ destabilizes $E$. 

So this is a special case of our set-up with $r \,=\, 1 \,=\, \text{rank}(E/E_1)$.
By \eqref{e3}, we have $\theta \,=\, \text{degree}(E/E_1) =
\text{degree}(E) < 0$. Further note that the hypothesis in Theorem
\ref{sc2} is satisfied. In fact, we have an equality 
$\mu(E_2/E_1)-\mu(E_1/E_0) \,= \,\theta$. So Theorem \ref{sc2} applies
and we obtain precise values of Seshadri constants for any ample line
bundle on $\gr$. 

\end{remark}

We now give four examples which satisfy the hypotheses of either Theorem
\ref{sc1} or Theorem \ref{sc2}. As far as we know, 
in all four examples our results give new computations of Seshadri constants. 
In our examples, $\text{rank}(E)=4$.

\begin{example}\label{ex1}
Let $X$ be a smooth complex projective curve. Let $L$ be a line bundle
of degree -1 on $X$, and set $E \,=\, L \oplus \str_X^{\oplus 3}$. Then
${\rm degree}(E)\,=\,-1$ and $\mu(E) \,=\, -\frac{1}{4}$. Since $E$ has a subbundle of
degree zero, it is not semistable. The Harder-Narasimhan filtration of
$E$ is
$$0\,=\, E_0\, \subset\, E_1 \,=\, \str_X^{\oplus 3} \,\subset\, E_2 \,=\, E\, .$$ 
Let $r \,:=\, \text{rank}(E_2/E_1) \,=\, 1$. Note that
$\theta \,=\, \text{degree}(L) \,=\, -1$, and $E$ is
normalized as in \eqref{theta}. 

Since $\mu(E_2/E_1)-\mu(E_1/E_0) \,=\, -1 - 0 \,=\, \theta$, the hypotheses of
Theorem \ref{sc2} are satisfied. Thus this theorem precisely computes the
Seshadri constants of line bundles on $\gr$. 

Note that the sections of $\str_{\gr}(1)$ are given by
global sections of $E$, and there are three linearly independent global sections
of $E$ given by the three copies of $\str_X$ in $E$. These three
sections of $\str_{\gr}(1)$ meet precisely in $\Gamma_s$, which is the
section of the projection $\gr \,\longrightarrow\, X$ given by the rank 1 quotient $E
\,\longrightarrow\, L$. So $\Gamma_s$ is the only curve in $\gr$ which meets
$\str_{\gr}(1)$ negatively. In this example, this can also be seen by
observing that $E \,\longrightarrow\, L$ is the only rank 1 quotient of $E$ which has
negative degree. 
\end{example}

\begin{example}\label{ex2}
Let $X$ be a smooth complex projective curve. Let $L$ be a line bundle
of degree $-1$ on $X$, and set $E \,=\, L^{\oplus 2}\oplus \str_X^{\oplus
2}$. Then $\mu(E) \,=\, -\frac{1}{2}$, and $E$ is not semistable because it has
a subbundle of degree zero. The Harder-Narasimhan filtration of $E$ is 
$$0\,=\, E_0\, \subset\, E_1 \,=\, \str_X^{\oplus 2} \,\subset\, E_2 \,=\, E\, .$$ 
Let $r \,:=\, \text{rank}(L^{\oplus 2}) \,=\, 2$. Note that
$\theta \,=\, \text{degree}(L^{\oplus 2}) \,=\, -2$, and $E$ is
normalized as in \eqref{theta}. 

Since $\mu(E_2/E_1)-\mu(E_1/E_0) \,=\, -1 - 0 \,>\, \theta$, the hypotheses of
Theorem \ref{sc2} are not satisfied. Indeed, $\gr$ contains curves
which meet $\str_{\gr}(1)$ negatively, but are different from
$\Gamma_s$. For example, the section of $\gr \,\longrightarrow\, X$ corresponding to
the rank 2 quotient $E \,\longrightarrow\, \str_X \oplus L$ has intersection -1 with
$\str_{\gr}(1)$. 

On the other hand, the hypotheses of Theorem \ref{sc1} are satisfied, because 
$\text{rank}(E_1) \,=\, 2$ and $\text{degree}(E_1) \,= 0$. Hence by Theorem 
\ref{sc1}, we can find precise values of Seshadri constants of ample line bundles 
on $\gr$ at points not contained in the base locus of $\str_{\gr}(1)$. In fact, we 
have precise values at points not contained in curves which meet $\str_{\gr}(1)$ 
negatively.
\end{example}

\begin{example}\label{ex3}
Let $X$ be a smooth complex projective curve. Let $L_{-1}$ be a line bundle of 
degree $-1$ on $X$, and let $L_1$ be a line bundle of degree $1$. Set $E \,=\, 
L_{-1}^{\oplus 2}\oplus L_1^{\oplus 2}$. Then $\mu(E) \,=\, 0$, and $E$ is not 
semistable because it has a sub-line-bundle of degree 1. The Harder-Narasimhan 
filtration of $E$ is $$0\,=\, E_0\, \subset\, E_1 \,=\, L_1^{\oplus 2} 
\,\subset\, E_2 \,=\, E\, .$$ Let $r \,:=\, \text{rank}(E_2/E_1) \,=\, 
\text{rank}(L_{-1}^{\oplus 2}) \,=\, 2$. Note that $$\theta\,=\, \text{degree}(E_2/E_1)
\,=\, \text{degree}(L_{-1}^{\oplus 2}) \,=\, -2\, ,$$ and $E$ is normalized as in
\eqref{theta}.

Since $\mu(E_2/E_1)-\mu(E_1/E_0) \,=\, -1 - 1 \,= \,\theta$, the hypotheses of
Theorem \ref{sc2} are satisfied. Thus we can precisely compute the
Seshadri constants of line bundles on $\gr$. 
\end{example}

\begin{example}\label{ex4}
Let $X$ be a smooth complex projective curve of positive genus. Let $L_{-1}$ be a 
line bundle of degree $-1$ on $X$, and let $L_0$ be a non-trivial line bundle of 
degree $0$. Set $E \,=\, L_{-1}\oplus \str_X^{\oplus 2} \oplus L_0$. Then $\mu(E)\,=\, 
-\frac{1}{4}$ and $E$ is not semistable because it has subbundles of degree $0$. The 
Harder-Narasimhan filtration of $E$ is $$0\,=\, E_0\, \subset\, E_1 \,=\, 
\str_X^{\oplus 2} \oplus L_0 \,\subset\, E_2 \,=\, E\, .$$ Let $r \,:=\, 
\text{rank}(E_2/E_1) \,=\, \text{rank}(L_{-1}) \,=\, 1$. Note that $\theta \,=\, 
\text{degree}(E_2/E_1) \,=\, -1$, and $E$ is normalized as in \eqref{theta}.

Since $\mu(E_2/E_1)-\mu(E_1/E_0) \,=\, -1 - 0 \,=\, \theta$, the hypotheses of
Theorem \ref{sc2} are satisfied. Thus we can precisely compute the
Seshadri constants of line bundles on $\gr$. 

Unlike in Example \ref{ex1}, the line bundle $\str_{\gr}(1)$ has only two 
independent sections. So the base locus of $\str_{\gr}(1)$ strictly contains 
$\Gamma_s$, which is the section of $\gr\,\longrightarrow\, X$ corresponding to 
the quotient $E \,\longrightarrow\, L_{-1}$. However, note that $\Gamma_s$ is the 
only curve in $\gr$ which meets $\str_{\gr}(1)$ negatively, by Corollary 
\ref{cor2}.
\end{example}

\section*{Acknowledgements}

We thank the referee for helpful comments to improve the exposition.
The first three authors thank the International Centre for Theoretical Sciences at
Bangalore for hospitality. The first author is partially supported by a J. C. Bose
Fellowship. The second author is partially supported by a grant from Infosys
Foundation.


\end{document}